\documentclass[12pt]{amsart}
\usepackage{amssymb}
\usepackage{amsthm}
\usepackage{amsmath}
\usepackage{enumerate}
\usepackage[colorlinks]{hyperref}
\usepackage{mathrsfs}
\usepackage{booktabs}
\usepackage{marginnote}

\theoremstyle{plain}

\newtheorem{theorem}{Theorem}
\newtheorem{lemma}[theorem]{Lemma}
\newtheorem{proposition}[theorem]{Proposition}

\newtheorem{cor}[theorem]{Corollary}

\newtheorem{defin}[theorem]{Definition}
\newtheorem{prop}[theorem]{Proposition}

\newcommand{\NN}{\mathbb N}

\newcommand{\RR}{\mathbb R}

\newcommand{\irr}{\textup{Irr}}

\newcommand{\ra}{\rightarrow}

\newcommand{\s}{\scriptstyle}

\renewcommand{\s}{\mathscr}
\newcommand{\la}{\lambda}
\newcommand{\G}{\Gamma}

\newcommand{\Al}{\textup{\textsf{A}}}
\newcommand{\Sy}{\textup{\textsf{S}}}

\newcommand{\twolineindex}[2]{
  \begin{array}{c}
    \\[-15pt]{\scriptstyle #1}
    \\[-3pt]{\scriptstyle #2}
  \end{array}
}

\begin{document}
\title[Character degrees of the symmetric groups]
{The largest character degrees of the symmetric and alternating
groups}

\thanks{The research of the first author leading to these results has
  received funding from the European Union's Seventh Framework
  Programme (FP7/2007-2013) under grant agreement no. 318202, from ERC
  Limits of discrete structures Grant No.\ 617747 and from OTKA
  K84233. The third author is partially supported by NSA Young
  Investigator Grant \#H98230-14-1-0293 and a BCAS Faculty Scholarship
  Award from the Buchtel College of Arts and Sciences-The University
  of Akron}

\author[Z. Halasi]{Zolt\'an Halasi} \address{Department of Algebra
and Number Theory, Institute of Mathematics, University of Debrecen,
4010, Debrecen, Pf.~12, Hungary} \email{halasi.zoltan@renyi.mta.hu}

\author[C. Hannusch]{Carolin Hannusch} \address{Department of Algebra
and Number Theory, Institute of Mathematics, University of Debrecen,
4010, Debrecen, Pf.~12, Hungary}
\email{carolin.hannusch@science.unideb.hu}

\author[H.\,N. Nguyen]{Hung Ngoc Nguyen}
\address{Department of Mathematics, The University of Akron, Akron,
Ohio 44325, USA} \email{hungnguyen@uakron.edu}



\date{\today}

\begin{abstract} We show that the largest character degree of an
alternating group $\Al_n$ with $n\geq 5$ can be bounded in terms of
smaller degrees in the sense that
\[
b(\Al_n)^2<\hspace{-10pt}\sum_{\twolineindex{\psi\in\irr(\Al_n)}{\psi(1)<
b(\Al_n)}}\hspace{-10pt}\psi(1)^2,
\]
where $\irr(\Al_n)$ and $b(\Al_n)$
respectively denote the set of irreducible complex characters of
$\Al_n$ and the largest degree of a character in $\irr(\Al_n)$. This
confirms a prediction of I.\,M.~Isaacs for the alternating groups
and answers a question of M.~Larsen, G.~Malle, and P.\,H.~Tiep.
\end{abstract}

\maketitle


\section{Introduction}

For a finite group $G$, let $\irr(G)$ and $b(G)$ respectively denote
the set of irreducible complex characters of $G$ and the largest
degree of a character in $\irr(G)$, then set
\[\varepsilon(G):=\frac{\sum_{\chi\in\irr(G),\,\chi(1)<b(G)}\chi(1)^2}{b(G)^2}.\]
Since $b(G)$ divides $|G|$ and $b(G)^2\leq |G|$, one can write
$|G|=b(G)(b(G)+e)$ for some non-negative integer $e$. The
(near-)extremal situations where $b(G)$ is very close to
$\sqrt{|G|}$, or equivalently $e$ is very small, have been studied
considerably in the literature, see~\cite{Berkovich,Snyder}.
According to the result of Y.~Berkovich~\cite{Berkovich} which says
that $e=1$ if and only if $G$ is either an order 2 group or a
$2$-transitive Frobenius group, there is no upper bound for $|G|$ in
this case. On the other hand, when $e>1$, N.~Snyder~\cite{Snyder}
showed that $|G|$ is bounded in terms of $e$ and indeed $|G|\leq
((2e)!)^2$.

In an attempt to replace Snyder's factorial bound with a polynomial
bound of the form $Be^6$ for some constant $B$, Isaacs \cite{Isaacs}
raised the question whether the largest character degree of a
non-abelian simple group can be bounded in terms of smaller degrees
in the sense that $\varepsilon(S)\geq \varepsilon$ for some
universal constant $\varepsilon>0$ and for all non-abelian simple
groups $S$. Answering Isaacs's question in the affirmative, Larsen,
Malle, and Tiep~\cite{Larsen-Malle-Tiep} showed that the bounding
constant $\varepsilon$ can be taken to be $2/(120\,000!)$. We note
that this rather small bound comes from the alternating groups,
see~\cite[Theorem~2.1 and Corollary~2.2]{Larsen-Malle-Tiep} for more
details.

To further improve Snyder's bound from $Be^6$ to $e^6+e^4$, Isaacs
even predicted that $\varepsilon(S)> 1$ for every non-abelian simple
group $S$. This was in fact confirmed in~\cite{Larsen-Malle-Tiep}
for the majority of simple classical groups, and for all simple
exceptional groups of Lie type as well as sporadic simple groups.
Therefore, Larsen, Malle and Tiep questioned whether one can improve
the bound $2/(120\,000!)$ for the remaining non-abelian simple
groups -- the alternating groups $\Al_n$ of degree at least 5.
Though Snyder's bound has been improved significantly by different
methods in recent works of C.~Durfee and S.~Jensen
\cite{Durfee-Jensen} and M.\,L.~Lewis~\cite{Lewis}, Isaacs's
prediction and in particular Larsen-Malle-Tiep's question are still
open.

In this paper we are able to show that $\varepsilon(\Al_n)>1$ for
every $n\geq 5$.

\begin{theorem}\label{theorem-main-1}
  For every integer $n\geq 5$,
  \[
  \sum_{\twolineindex{\psi\in\irr(\Al_n)}{\psi(1)< b(\Al_n)}}\hspace{-10pt}
  \psi(1)^2>b(\Al_n)^2.
  \]
\end{theorem}

Unlike the simple groups of Lie type where one can use Lusztig's
classification of their irreducible complex characters, it seems
more difficult to work with the largest character degree of the
alternating groups. For instance, while $b(S)$ is known for $S$ a
simple exceptional groups of Lie type or a simple classical group
whose underlying field is sufficiently large
(see~\cite{Seitz,Larsen-Malle-Tiep}), $b(\Al_n)$ as well as
$b(\Sy_n)$ are far from determined. We note that the current best
bound for $b(\Sy_n)$ is due to A.\,M.~Vershik and
S.\,V.~Kerov~\cite{Vershik-Kerov}.

It is clear that $b(\Sy_n)/2\leq b(\Al_n)\leq b(\Sy_n)$ and as we
will prove in Section~\ref{section3}, indeed
$b(\Sy_n)/2<b(\Al_n)\leq b(\Sy_n)$ is always the case. As far as we
know, it is still unknown for what $n$ the equality
$b(\Al_n)=b(\Sy_n)$ actually occurs. It would be interesting to
solve this. Though it appears at first sight that
$b(\Al_n)=b(\Sy_n)$ holds most of the time, computational evidence
indicates that $b(\Al_n)<b(\Sy_n)$ is true quite often.

When $\Al_n$ and $\Sy_n$ do have the same largest character degree,
Theorem~\ref{theorem-main-1} is indeed a direct consequence of a
similar but stronger inequality for the symmetric groups.

\begin{theorem}\label{theorem-main-2} For every integer $n\geq 7$,
\[
\sum_{\twolineindex{\chi\in\irr(\Sy_n)}{\chi(1)< b(\Sy_n)}}
\hspace{-10pt}\chi(1)^2>2b(\Sy_n)^2.
\]
\end{theorem}

Our ideas to prove Theorems~\ref{theorem-main-1}
and~\ref{theorem-main-2} are different from those
in~\cite{Larsen-Malle-Tiep} and are described briefly as follows. We
first introduce a graph with the partitions of $n$ as vertices and a
partition $\lambda=(\lambda_1\geq\lambda_2\geq\cdots \geq
\lambda_k)$ is connected by an edge to $\lambda_{up}:=(\la_1+1\geq
\la_2\geq\ldots\geq \la_{k-1})$ only when $\la_k=1$ and to
$\la_{dn}:=(\la_1-1\geq\la_2\geq \ldots\geq\la_k\geq 1)$ only when
$\la_1>\la_2$. It turns out that if $\lambda$ corresponds to an
irreducible character of $\Sy_n$ of the largest degree, then
$\lambda$ has precisely two neighbors in this graph. Furthermore,
the degrees of the characters corresponding to $\lambda_{up}$ and
$\lambda_{dn}$ are shown to be `close' to that corresponding to
$\lambda$, see Lemma~\ref{proposition-inequality}. With this in
hand, we deduce that $\Sy_n$ has at least as many irreducible
characters of degree close to but smaller than $b(\Sy_n)$ as those
of degree $b(\Sy_n)$, and therefore Theorem~\ref{theorem-main-2}
holds when the largest character degree $b(\Sy_n)$ has large enough
multiplicity. When this multiplicity is smaller, we consider the
irreducible constituents of the induced character
$(\chi\hspace{-3pt}\downarrow_{\Sy_{n-1}})^{\Sy_n}$ where
$\chi\in\irr(\Sy_n)$ is a character of degree $b(\Sy_n)$ and observe
that there are enough constituents of degree smaller than $b(\Sy_n)$
to prove the desired inequality.

As mentioned already, Theorem~\ref{theorem-main-1} follows from
Theorem~\ref{theorem-main-2} in the case $b(\Al_n)=b(\Sy_n)$.
However, the other case $b(\Al_n)<b(\Sy_n)$ creates some
difficulties. To handle this, we reduce the problem to the situation
where $\Sy_n$ has precisely one irreducible character of degree
$b(\Sy_n)$ and the second largest character degree equal to
$b(\Al_n)$. We then work with the multiplicity of degree $b(\Al_n)$
and follow similar but more delicate arguments than in the case
$b(\Al_n)=b(\Sy_n)$.

Following the ideas outlined above, we can also prove the following

\begin{theorem}\label{theorem-main-3} We have
$\varepsilon(\Al_n)\rightarrow\infty$ and
$\varepsilon(\Sy_n)\rightarrow\infty$ as $n\rightarrow\infty$.
\end{theorem}
\noindent This convinces us to believe that
$\varepsilon(S)\rightarrow\infty$ as $|S|\rightarrow\infty$ for all
non-abelian simple groups $S$ and it would be interesting to confirm
this.

The paper is organized as follows. In the next section, we give a
brief summary of the character theory of the symmetric and
alternating groups. The graph on partitions and relevant results are
presented in Section~\ref{section-graph}. Section~\ref{section3} is
devoted to the proofs of Theorems~\ref{theorem-main-1}
and~\ref{theorem-main-2} and finally Theorem~\ref{theorem-main-3} is
proved in Section~\ref{section5}.


\section{Preliminaries}

 For the reader's convenience and to introduce notation, we briefly summarize some basic facts on the representation theory of the
symmetric and alternating groups.

We say that a finite sequence $\lambda:=(\lambda_1,\lambda_2, \ldots
,\lambda_k)$ is a partition of $n$ if
$\lambda_1\geq\lambda_2\geq\cdots \geq\lambda_k$ and
$\lambda_1+\lambda_2+\cdots+\lambda_k=n$. The Young diagram
corresponding to $\lambda$, denoted by $Y_\lambda$, is defined to be
the finite subset of $\NN\times\NN$ such that
\[(i,j)\in Y_\lambda \text{ if and only if } i\leq \lambda_j.\] The conjugate partition of
$\lambda$, denoted by $\overline{\lambda}$, is the partition whose
associated Young diagram is obtained from $Y_\lambda$ by reflecting
it about the line $y=x$. So $\lambda=\overline{\lambda}$ if and only
if $Y_\lambda$ is symmetric and in that case we say that $\lambda$
is self-conjugate.

For each node $(i,j)\in Y_\lambda$, the so-called \emph{hook length}
$h(i,j)$ is defined by
\[h(i,j):=1+\lambda_j+\overline{\lambda}_i-i-j.\]
That is, $h(i,j)$ is the number of nodes that are directly above it,
directly to the right of it, or equal to it. The \emph{hook-length
product} of $\lambda$ is then defined by
\[
H(\lambda):=\prod_{(i,j)\in Y_\lambda}h_\la(i,j).
\]

For each positive integer $n$, it is known that there is a
one-to-one correspondence between the irreducible complex characters
of the symmetric group $\Sy_n$ and the partitions of $n$. We denote
by $\chi_\lambda$ the irreducible character of $\Sy_n$ corresponding
to $\lambda$. The degree of $\chi_\lambda$ is given by the
\emph{hook-length formula}, see~\cite{Frame-Robinson-Thrall}:
\[\chi_{\lambda}(1)=\frac{n!}{H(\lambda)}.\]

The irreducible characters of $\Al_n$ can be obtained by restricting
those of $\Sy_n$ to $\Al_n$. More explicitly, if $\lambda$ is not
self-conjugate then $\chi_{\lambda}\hspace{-3pt}\downarrow_{\Al_n}=
\chi_{\overline{\lambda}}\hspace{-3pt}\downarrow_{\Al_n}$ is
irreducible and otherwise,
$\chi_{\lambda}\hspace{-3pt}\downarrow_{\Al_n}$ splits into two
different irreducible characters of the same degree. Therefore, the
degrees of the irreducible characters of $\Al_n$ are
$$ \left\{\begin
{array}{ll}
\chi_\lambda(1) & \text{ if } \lambda\neq\overline{\lambda},\\
\chi_\lambda(1)/2 & \text{ if } \lambda=\overline{\lambda}.
\end {array} \right.$$

For each partition $\lambda$ of $n$, let $A(\lambda)$ and
$R(\lambda)$ denote the sets of nodes that can be respectively added
or removed from $Y_\lambda$ to obtain another Young diagram
corresponding to a certain partition of $n+1$ or $n-1$ respectively.
As shown in \cite[page~67]{Larsen-Malle-Tiep}, we have
$|A(\lambda)|^2-|A(\lambda)|\leq 2n$, and hence
\[A(\lambda)\leq \frac{1+\sqrt{1+8n}}{2}.\]
Similarly, we have $|R(\lambda)|^2+|R(\lambda)|\leq 2n$ and
\[R(\lambda)\leq \frac{-1+\sqrt{1+8n}}{2}.\]

The well-known branching rule (see~\cite[\S9.2]{James} for instance)
asserts that the restriction of $\chi_\lambda$ to $\Sy_{n-1}$ is a
sum of irreducible characters of the form $\chi_{Y_\lambda
\backslash \{ (i,j) \}}$ as $(i,j)$ goes over all nodes in
$R(\lambda)$. Also, by Frobenius reciprocity, the induction of
$\chi_\lambda$ to $\Sy_{n+1}$ is a sum of irreducible characters of
the form $\chi_{Y_\lambda \cup \{(i,j)\}}$ as $(i,j)$ goes over all
nodes in $A(\lambda)$.

It follows from the branching rule that the number of irreducible
constituents of the induced characters
$(\chi_\lambda\hspace{-3pt}\downarrow_{\Sy_{n-1}})^{\Sy_n}$ is at
most
\[
\frac{-1+\sqrt{1+8n}}{2}\cdot \frac{1+\sqrt{1+8(n-1)}}{2}.
\]
In particular, this number is smaller than $2n$.


\section{A graph on partitions}\label{section-graph}

Let $\s P$ denote the set of partitions of $n$. Furthermore, let
\[b_1=b(\Sy_n)>b_2>\ldots>b_m=1\] be the distinct character degrees of $\Sy_n$.
For every $1\leq i\leq m$ let
\[
\s M_i:=\{\la\in\s P\,|\,\chi_\la(1)=b_i\}
\]
so that
\[
\s P=\s M_1\cup \s M_2\cup\ldots\cup \s M_m\quad (\textrm{disjoint
union}).
\]
\begin{defin}
  For a partition $\la=(\lambda_1\geq \lambda_2\geq \ldots \geq
  \la_k)$ we define partitions $\la_{dn}$ and $\la_{up}$ in the
  following way. The partition $\la_{dn}$ is defined only if
  $\la_1>\la_2$ and in this case let $\la_{dn}:=(\la_1-1\geq\la_2\geq
  \ldots\geq\la_k\geq 1)$. Similarly, the partition $\la_{up}$ is
  defined only if $\la_k=1$ and in this case let $\la_{up}:=
  (\la_1+1\geq \la_2\geq\ldots\geq \la_{k-1})$.
\end{defin}
Next we define a graph on $\s P$.
\begin{defin}
  Let $\G=(V,E)$ be the graph with vertex set $V=\s P$ and edge set
  $E=\{(\la,\mu)\,|\,\mu=\la_{dn}\textrm{ or }\mu=\la_{up}\}$.
  Furthermore, let $\G_{\s M_i}$ be the induced subgraph of $\G$ on $\s M_i$.
\end{defin}
For each vertex $\la\in V$, let $d(\la)$ denote the degree of $\la$,
that is, the number of vertices that are connected to $\la$ by an
edge of $\G$. It is clear that $d(\la)\leq 2$ for every $\la\in V$.
Moreover, every connected component of $\G$ is a simple path.
\begin{lemma}\label{lemma-d(lambda)=2}
  For every $1\leq r\leq m$ we have
  \[
  |\{\la\in \s M_r\,|\,d(\la)<2\}|\leq 2|\cup_{i<r}\s M_i|.
  \]
  In particular, we have the following
  \begin{enumerate}
  \item $d(\la)=2$ for all partitions $\la\in \s M_1$.
  \item If $|\s M_1|=1$, then $d(\la)=2$ for all but at most two
  partitions
    $\la\in \s M_2$.
  \end{enumerate}
\end{lemma}
\begin{proof}
  First we prove that
  \[
  |\{\la\in \s M_r\,|\,\nexists\;\la_{dn}\}|\leq |\cup_{i<r}\s M_i|.
  \]
  Let $\la=(\la_1\geq\la_2\geq \ldots\geq \la_k)\in \s M_r$ such that
  $\la_1=\la_2=\ldots=\la_s=t$ but $\la_{s+1}<t$ for some $1< s\leq
  k$. Let
  $\la_{\ra 1}:=(\lambda_1+1,\lambda_2,...,\lambda_s-1,\lambda_{s+1},...,\lambda_k)$
  and $x_j:=h_\la(s,j)$ for every $1\leq j\leq t-1$.  Calculating the
  ratio of the hook-length products $H(\la_{\ra 1})$ and $H(\la)$ we get
  \begin{align*}
  \frac{H(\la_{\ra 1})}{H(\la)}&=
    \frac{\prod_{j=1}^{t-1}(h_{\la_{\ra 1}}(s,j)h_{\la_{\ra 1}}(1,j))
    \prod_{i=2}^{s-1}h_{\la_{\ra 1}}(i,t)}
    {\prod_{j=1}^{t-1}(h_{\la}(s,j)h_{\la}(1,j))\prod_{i=2}^{s-1}h_{\la}(i,t)}\\
    &=\frac{\prod_{j=1}^{t-1}((x_j-1)(x_j+s))\cdot (s-2)!}
    {\prod_{j=1}^{t-1}(x_j(x_j+s-1))\cdot (s-1)!}\\
    &=\frac{1}{s-1}\prod_{j=1}^t\left(1-\frac{s}{x_j(x_j+s-1)}\right)<1.
  \end{align*}
  Hence for the degrees of characters we get
  \[\frac{\chi_\la(1)}{\chi_{\la_{\ra 1}(1)}}=\frac{H(\la_{\ra 1})}{H(\la)}<1.\]
  Thus, we have defined a map $\la\mapsto \la_{\ra 1}$
  from the set
  $\{\la\in \s M_r\,|\,\nexists\;\la_{dn}\}$ into $\cup_{i<r}\s M_i$.
  This map is clearly injective, so
  \[
  |\{\la\in \s M_r\,|\,\nexists\;\la_{dn}\}|\leq |\cup_{i<r}\s M_i|
  \]
  follows. The dual map  $\la\mapsto \overline{\la}$ defines a
  bijection between $\{\la\in \s M_r\,|\,\nexists\;\la_{dn}\}$ and
  $\{\la\in \s M_r\,|\,\nexists\;\la_{up}\}$. It follows that
  \[
  |\{\la\in \s M_r\,|\,\nexists\;\la_{up}\}|\leq |\cup_{i<r}\s M_i|.
  \]
  Therefore,
  \[
  |\{\la\in \s M_r\,|\,d(\la)<2\}|
  \leq|\{\la\in \s M_r\,|\,\nexists\;\la_{dn}\}|+
  |\{\la\in \s M_r\,|\,\nexists\;\la_{up}\}|
  \leq 2|\cup_{i<r}\s M_i|
  \]
  and the proof is complete.
\end{proof}
\begin{lemma}\label{proposition-inequality}
  If $d(\la)=2$ then \[1<\frac{H(\la_{dn})H(\la_{up})}{H(\la)^2}<4.\]
\end{lemma}
\begin{proof}
  Let $\la=(\la_1> \la_2\geq \ldots\geq \la_k=1)\in \s P$ with
  $d(\la)=2$. Furthermore, let $x_j:=h(1,j)$ for $2\leq j\leq \la_1-1$
  and $y_i:=h(i,1)$ for $2\leq i\leq k-1$. Then we have
  $x_2>x_3>\ldots>x_{\la_1-1}\geq 2$ and $y_2>y_3>\ldots>y_{k-1}\geq
  2$.

  Calculating the ratios $H(\la_{dn})/H(\la)$ and
  $H(\la_{up})/H(\la)$ we obtain
  \[
  \frac{H(\la_{dn})}{H(\la)}=
  \prod_{j=2}^{\la_1-1}\frac{x_j-1}{x_j}\cdot
  2\prod_{i=2}^{k-1}\frac{y_i+1}{y_i}\]
  and
  \[\frac{H(\la_{up})}{H(\la)}=
  2\prod_{j=2}^{\la_1-1}\frac{x_j+1}{x_j}\cdot \prod_{i=2}^{k-1}\frac{y_i-1}{y_i}.
  \]
  It follows that
  \[
  \frac{H(\la_{dn})H(\la_{up})}{H(\la)^2}=
  4\prod_{j=2}^{\la_1-1}\frac{x_j^2-1}{x_j^2}\prod_{i=2}^{k-1}
  \frac{y_i^2-1}{y_i^2}.
  \]
  The right hand side of this inequality is clearly smaller than 4.
  Regarding the lower bound, we argue as follows. First, since the
  hook lengths $x_i$ are different integers bigger than $1$, we have
  \[
  2\prod_{j=2}^{\la_1-1}\frac{x_j^2-1}{x_j^2}>2\prod_{m=2}^{\infty}
  \frac{(m-1)(m+1)}{m^2}=1.
  \]
  The same can be said about $2\prod_{i=2}^{k-1}
  \frac{y_i^2-1}{y_i^2}$ and so their product is also bigger than
  $1$. The proof is complete.
\end{proof}
Using the previous lemma, we can show that $\Sy_n$ has many
irreducible characters of degree close to but smaller than
$b(\Sy_n)$.
\begin{prop}\label{proposition-cardinalityofM}
  For every $1\leq r\leq m$ we have
  \[
  \left|\left\{\mu\in\s P\,\Big|\,
  \frac{b_r}{4}< \chi_\mu(1)< b_r\right\}\right|\geq|\s M_r|-4|\cup_{i<r}\s M_i|.
  \]
  In particular, we have
  \begin{enumerate}
    \item
      \[
      \left|\left\{\mu\in\s P\,\Big|\,
      \frac{b_1}{4}< \chi_\mu(1)< b_1\right\}\right|\geq|\s M_1|.
      \]
    \item
      If $|\s M_1|=1$, then
      \[
      \left|\left\{\mu\in\s P\,\Big|\,
      \frac{b_2}{4}< \chi_\mu(1)< b_2\right\}\right|\geq|\s M_2|-4.
      \]
  \end{enumerate}
\end{prop}
\begin{proof}
  For a real-valued function $f:V\mapsto\RR$ defined on the vertex set
  of the graph $\G=(V,E)$ we say that $x\in V$ is a local maximum
  (resp. minimum) of $f$ if $f(y)\leq f(x)$ (resp. $f(y)\geq f(x)$)
  for every $(x,y)\in E$.

  Let $C$ be any connected component of $\G$, so $C$ is a simple
  path. We note that if $d(\la)=2$ then either $H(\la)<H(\la_{dn})$ or
  $H(\la)<H(\la_{up})$ by
  Lemma~\ref{proposition-inequality}. Therefore, there is no local
  maximum $\la\in C$ of the hook-length product function $H:C\mapsto
  \NN$ with $d(\la)=2$. It follows that if $\la,\mu\in C$ are both
  local minimums of $H$ on $C$, then $H$ is constant on the subpath
  connecting $\la$ and $\mu$. Furthermore, the restriction of $H$ to a
  subpath of $C$ of length $\geq 3$ cannot be constant, since for an
  inner point $\la$ of such a subpath we would have
  \[\frac{H(\la_{dn})H(\la_{up})}{H(\la)^2}=1\] and this violates the
  inequality in Lemma~\ref{proposition-inequality}.

  It follows from this argument that $|C\cap \s M_r|\leq
  2$. Furthermore, if $|C\cap \s M_r|=2$, then the two vertices of
  $C\cap\s M_r$ are either neighboring vertices in $\G$ or all the
  inner points of the subpath connecting them are elements from the
  set $\cup_{i<r}\s M_i$.
  This implies that
  \[
  |\{\la\in\s M_r\,|\,d(\la)=2,\ \min(H(\la_{dn}),H(\la_{up}))<H(\la)\}|
  \leq 2|\cup_{i<r}\s M_i|.
  \]

  Let \[X:=\{\la\in\s M_r\,|\,d(\la)=2,\ \min(H(\la_{dn}),H(\la_{up}))\geq
  H(\la)\}.\]
  Taking also the result of Lemma~\ref{lemma-d(lambda)=2} into account
  we deduce that  \[|X|\geq |\s M_r|-4|\cup_{i<r}\s M_i|.\]

  Now, for every $\la\in X$ we will associate a $\mu\in\s P$
  such that $(\la,\mu)\in E$ and
  $\chi_\la(1)/4<\chi_\mu(1)<\chi_\la(1)$. Let $C$ be the component of $\G$
  containing $\la$. If $\la$ is the only vertex of $C\cap X$
  then \[1<\frac{H(\la_{dn})}{H(\la)},\frac{H(\la_{up})}{H(\la)}<4\]
  so that both $\mu=\la_{dn}$ and $\mu=\la_{up}$ are good choices.  On
  the other hand, if $|C\cap  X|=2$, then $|\{\la_{dn},\la_{up}\}\cap
   X|=1$ and we just choose $\mu$ to be the vertex in
  $\{\lambda_{dn},\lambda_{up}\}$ that is not in $X$.

  It remains to prove that the function $\la\mapsto \mu$ we have just defined is
  injective. But this follows from the fact that disjoint elements of
  $X$ cannot have a common neighbor in $\G$.
\end{proof}


\section{Theorems \ref{theorem-main-1} and~\ref{theorem-main-2}}\label{section3}

We now show that Proposition \ref{proposition-cardinalityofM}
implies Theorem~\ref{theorem-main-2} when the cardinality of $\s
M_1$ is large enough.
\begin{cor}\label{corollary-Theorem2}
  If $|\s M_1|\geq 32$, then Theorem~\ref{theorem-main-2} holds.
\end{cor}
\begin{proof}
  Let
  \[
  \s T:=\left\{\chi\in\irr(\Sy_n)\,\Big|\,
  \frac{b(\Sy_n)}{4}<\chi(1)<b(\Sy_n)\right\}.
  \]
  By Proposition~\ref{proposition-cardinalityofM} (1) we have $|\s T|\geq
  |\s M_1|\geq 32$.  Thus,
  \[
  \sum_{\twolineindex{\chi\in\irr(\Sy_n)}{\chi(1)\neq b(\Sy_n)}}\chi(1)^2\geq
  \sum_{\chi\in \s T}\chi(1)^2>|\s T|\cdot \left(\frac{b(\Sy_n)}{4}\right)^2
    \geq   2b(\Sy_n)^2,
  \] as desired.
\end{proof}

The case where $|\s M_1|$ is small is handled by a different
technique. From now on, for characters $\chi_1,\chi_2$ of a group
$G$ we write $[\chi_1,\chi_2]$ to denote their inner product.
\begin{prop}\label{proposition-Theorem2}
  Let $\lambda\in\s M_1$ and let $\chi:=\chi_\la$. If $n\geq 50$ and $|\s
  M_1|\leq 31$, then
  \[
  \sum_{\twolineindex{[\varphi,(\chi\hspace{-1pt}\downarrow_{S_{n-1}})^{\Sy_n}]
      \neq 0}{\varphi(1)<b(\Sy_n)}}
  \hspace{-15pt}\varphi^2(1)>2b(\Sy_n)^2.
  \]
In particular, Theorem~\ref{theorem-main-2} holds in this case.
\end{prop}
\begin{proof}
  By the branching rule we have
  \[
  (\chi\hspace{-3pt}\downarrow_{\Sy_{n-1}})^{\Sy_n}=
  |R(\la)|\cdot \chi+\sum_{i\neq j}\chi_{\la_{i\to j}},
  \]
  where we recall that $R(\la)$ is the set of nodes that can be
  removed from $Y_\lambda$ to obtain another Young diagram of size
  $n-1$, and $\la_{i\to j}$ denotes the partition obtained from
  $\lambda$ by moving the last node from row $i$ to the end of the row
  $j$.

  We also recall that $|R(\lambda)|\leq \frac{-1+\sqrt{1+8n}}{2}$ and
  if $\mu$ is a partition of $n-1$ then
  $|A(\mu)|<\frac{1+\sqrt{8n-7}}{2}$. Therefore the sum on the right
  hand side has at most
  \[
  \frac{-1+\sqrt{1+8n}}{2}\cdot \frac{1+\sqrt{8n-7}}{2} <2n
  \]
  characters.  Furthermore, $\chi$ appears at most
  $\frac{-1+\sqrt{1+8n}}{2}<\sqrt{2n}$ times, while there are at most
  $30$ other characters in this sum with degree $b(\Sy_n)$. Therefore,
  \[
  \sum_{\twolineindex{[\varphi,(\chi\hspace{-1pt}
      \downarrow_{S_{n-1}})^{\Sy_n}]\neq 0}{\varphi(1)<b(\Sy_n)}}
  \hspace{-10pt}\varphi(1)>(n-\sqrt{2n}-30)b(\Sy_n).
  \]
  Using the Cauchy-Schwarz inequality, we deduce that
  \[
  \sum_{\twolineindex{[\varphi,(\chi\hspace{-1pt}
      \downarrow_{S_{n-1}})^{\Sy_n}]\neq 0}{\varphi(1)<b(\Sy_n)}}
  \hspace{-10pt}\varphi^2(1)>\left(\frac{1}{\sqrt{2n}}
    (n-\sqrt{2n}-30)\right)^2\cdot b(\Sy_n)^2.
  \]
  It remains to check that
  \[
  \frac{1}{\sqrt{2n}}(n-\sqrt{2n}-30)\geq 1
  \]
but this is clear as $n\geq 50$.
\end{proof}

We are now ready to finish the proof of Theorem~\ref{theorem-main-2}.

\begin{proof}[Proof of Theorem~\ref{theorem-main-2}]
  In light of Corollary~\ref{corollary-Theorem2} and
  Proposition~\ref{proposition-Theorem2}, we only need to prove the
  theorem for $7\leq n\leq 49$. We have done that by computations in
  \cite{GAP} and the codes are available upon request.

  For each $n\leq 75$, partition corresponding to a character of
  $\Sy_n$ of the largest degree is available in~\cite{McKay}. Let $Y$
  be the Young diagram corresponding to this partition. We consider
  all possible Young diagrams obtained from $Y$ by moving one node
  from one row to another. For all those Young diagrams the degrees of
  the corresponding irreducible characters will be determined. If the
  degree of such a character coincides with the largest character
  degree of $\Sy_n$, then it will be excluded. We finally check that
  the sum of the squares of the remaining degrees is greater than
  $2b(\Sy_n)^2$, as desired.
\end{proof}


We now move on to a proof of Theorem~\ref{theorem-main-1}. First we
handle the case where $\Al_n$ and $\Sy_n$ have the same largest
character degree.
\begin{proposition}\label{proposition-theorem1}
  If $b(\Al_n)=b(\Sy_n)$ then Theorem~\ref{theorem-main-1} holds.
\end{proposition}
\begin{proof}
  Recall that the restriction of each irreducible character of $\Sy_n$
  to $\Al_n$ is either irreducible or a sum of two irreducible
  characters of equal degree. Therefore,
  \[
  \sum_{\twolineindex{\psi\in\irr(\Al_n)}{\psi(1)<b(\Al_n)}}
  \hspace{-10pt}\psi(1)^2\geq\frac{1}{2}
  \sum_{\twolineindex{\chi\in\irr(\Sy_n)}{\chi(1)<b(\Sy_n)}}
  \hspace{-10pt}\chi(1)^2.
  \]
  Using Theorem~\ref{theorem-main-2}, we obtain
  \[
  \sum_{\twolineindex{\psi\in\irr(\Al_n)}{\psi(1)< b(\Al_n)}}
  \hspace{-10pt}\psi(1)^2>b(\Sy_n)^2=b(\Al_n)^2,
  \]
  as desired.
\end{proof}


The proof of Theorem~\ref{theorem-main-1} in the case
$b(\Al_n)<b(\Sy_n)$ turns out to be more complicated. We will
explain this in the rest of this section.

Let $\lambda$ be the partition corresponding to a character of the
largest degree of $\Sy_n$. Then $\lambda$ is self-conjugate as
$b(\Al_n)<b(\Sy_n)$. Lemma~\ref{lemma-d(lambda)=2} guarantees that
$d(\lambda)=2$ and it follows that $\lambda_{up}$ and $\lambda_{dn}$
are \emph{not} self-conjugate. In particular,
$\chi_{\lambda_{up}}(1)$ and $\chi_{\lambda_{dn}}(1)$ are both at
most $b(\Al_n)$. Using Lemma~\ref{proposition-inequality}, we deduce
that
\[
\frac{b(\Sy_n)^2}{b(\Al_n)^2}=\frac{\chi_\lambda(1)^2}{b(\Al_n)^2}
\leq\frac{\chi_\lambda(1)^2}{\chi_{\lambda_{up}}(1)\chi_{\lambda_{dn}}(1)}<4,
\]
which in turns implies that $b(\Sy_n)/2 < b(\Al_n)$. In summary, we
have
\[\frac{b(\Sy_n)}{2} < b(\Al_n) < b(\Sy_n).\]

If there are two irreducible characters of $\Sy_n$ of the largest
degree, then the associated partitions are both self-conjugate and
so there are four irreducible characters of $\Al_n$ of degree
$b(\Sy_n)/2$, and we are done. So from now on we assume that there
is only one irreducible character of degree $b(\Sy_n)$ of $\Sy_n$.
In other words, $|\s M_1|=1$.  

If there is $\mu\in \s P$ such that $b(\Al_n)<\chi_\mu(1)<b(\Sy_n)$
then clearly $\mu$ must be self-conjugate. In this case $\Al_n$ has
two irreducible characters (lying under $\chi_\lambda$) of degree
$b(\Sy_n)/2$ and two irreducible characters (lying under $\chi_\mu$)
of degree at least $b(\Al_n)/2$, and we are done again. So we assume
furthermore that $b(\Al_n)$ is the second largest character degree
of $\Sy_n$, that is, $b(\Al_n)=b_2$.
\begin{proposition}\label{proposition-Theorem1-case2}
  Assume that there is precisely one irreducible character of $\Sy_n$
  of degree $b(\Sy_n)$ and $b(\Al_n)$ is the second largest character
  degree of $\Sy_n$. If $|\s M_2|\geq 20$, then
  Theorem~\ref{theorem-main-1} holds.
\end{proposition}

\begin{proof}
  Proposition~\ref{proposition-cardinalityofM} (2) and the
  hypothesis $|\s M_2|\geq 20$ imply that
  \[
  \Big|\Big\{\nu\in\s P\,\Big|\,
  \frac{b(\Al_n)}{4}< \chi_\nu(1)<
  b(\Al_n)\Big\}\Big|\geq 16.
  \]
  Thus the sum of the squares of the degrees of irreducible characters
  of $\Al_n$ lying under these characters $\nu$ is at least
  \[16\cdot
  \frac{1}{2}\left(\frac{b(\Al_n)}{4}\right)^2=\frac{b(\Al_n)^2}{2}.\]
  On the other hand, the sum of the squares of the degrees of the two
  characters of $\Al_n$ lying under $\chi_\lambda$ is $b(\Sy_n)^2/2$,
  which is larger than $b(\Al_n)^2/2$. So we conclude that
  \[
  \sum_{\twolineindex{\psi\in\irr(\Al_n)}{\psi(1)< b(\Al_n)}}
  \hspace{-10pt}\psi(1)^2>b(\Al_n)^2,
  \]
  as the theorem claimed.
\end{proof}


\begin{proposition}\label{proposition-Theorem1-case21}
  Assume that there is precisely one irreducible character of $\Sy_n$
  of degree $b(\Sy_n)$ and $b(\Al_n)$ is the second largest character
  degree of $\Sy_n$. Let $\mu\in \s M_2$ and let $\chi:=\chi_\mu$. If
  $n\geq 43$ and $|\s M_2|\leq 19$, then
  \[
  \sum_{\twolineindex{[\varphi,(\chi\hspace{-1pt}
      \downarrow_{\Sy_{n-1}})^{\Sy_n}]\neq 0}{\varphi(1)<b(\Al_n)}}
  \hspace{-10pt}\varphi^2(1)>2b(\Al_n)^2.
  \]
  In particular, Theorem~\ref{theorem-main-1} holds in this case.
\end{proposition}
\begin{proof}
  The proof goes along the same lines as that of
  Proposition~\ref{proposition-Theorem2} and so we will skip some
  details. First, by the branching rule,
  \[
  (\chi\hspace{-3pt}\downarrow_{\Sy_{n-1}})^{\Sy_n}=
  |R(\mu)|\cdot \chi+\sum_{i\neq j}\chi_{\mu_{i\to j}},
  \]
  where $R(\mu)$ is the set of nodes that can be removed from
  $Y_\mu$ to obtain another Young diagram of size $n-1$, and
  $\mu_{i\to j}$ denotes the partition obtained from $\mu$ by
  moving the last node from row $i$ to the end of the row $j$.

  In  the  sum  on  the  right  hand side,  there  are  at  most  $18$
  irreducible characters  (other than $\chi$)  with degree $b(\Al_n)$,
  and  at most one  irreducible character  with degree  $b(\Sy_n)$. We
  recall  that  $|R(\mu)|\leq  \frac{-1+\sqrt{1+8n}}{2}<\sqrt{2n}$
  and $b(\Sy_n)<2b(\Al_n)$. Therefore,
  \[
  \sum_{\twolineindex{[\varphi,(\chi\hspace{-1pt}
      \downarrow_{S_{n-1}})^{\Sy_n}]\neq 0}{\varphi(1)<b(\Al_n)}}
  \hspace{-10pt}\varphi(1)> (n-\sqrt{2n}-20)b(\Al_n).
  \]
  As the sum on the right hand side has at most $2n$ terms, the Cauchy-Schwarz
  inequality then implies that
  \[
  \sum_{\twolineindex{[\varphi,(\chi\hspace{-1pt}
      \downarrow_{S_{n-1}})^{\Sy_n}]\neq 0}{\varphi(1)<b(\Al_n)}}
  \hspace{-10pt}\varphi^2(1)> \left(\frac{1}{\sqrt{2n}}
    (n-\sqrt{2n}-20)\right)^2\cdot b(\Al_n)^2.
  \]
  Now the inequality in the proposition follows as
  $\frac{1}{\sqrt{2n}}(n-\sqrt{2n}-20)>\sqrt{2}$ when $n\geq 43$.

  To see that Theorem~\ref{theorem-main-1} holds under the given
  hypothesis, we just observe that
  \[
  \sum_{\twolineindex{\psi\in\irr(\Al_n)}{\psi(1)< b(\Al_n)}}\hspace{-10pt}\psi(1)^2>
  \frac{1}{2}\hspace{-15pt}\sum_{\twolineindex{[\varphi,(\chi\hspace{-1pt}
      \downarrow_{\Sy_{n-1}})^{\Sy_n}]\neq 0}{\varphi(1)<b(\Al_n)}}
  \hspace{-15pt}\varphi^2(1)>b(\Al_n)^2.
  \]
\end{proof}


Finally we can prove Theorem~\ref{theorem-main-1} in the case
$b(\Al_n)<b(\Sy_n)$.

\begin{proposition}\label{proposition-theorem11} If $b(\Al_n)<b(\Sy_n)$ then
Theorem~\ref{theorem-main-1} holds.
\end{proposition}

\begin{proof} As discussed at the beginning of this section, it
suffices to assume that there is precisely one irreducible character
of $\Sy_n$ of degree $b(\Sy_n)$ and $b(\Al_n)$ is the second largest
character degree of $\Sy_n$. Now the proposition follows from
Propositions~\ref{proposition-Theorem1-case2}
and~\ref{proposition-Theorem1-case21} when $n\geq 43$.

Let us now describe how we verify the theorem for $n<43$. As pointed
out earlier the partition $\lambda$ corresponding to the largest
degree in \cite{McKay} is self-conjugate. Denote the Young diagram
corresponding to this partition by $Y$, so $Y$ is symmetric. Then as
before we consider all Young diagrams obtained from $Y$ by moving
one node from one row to another. Note that all these Young diagrams
are not symmetric anymore and $Y_{up}$ and $Y_{dn}$ (the Young
diagrams of $\lambda_{up}$ and $\lambda_{dn}$) are among these
diagrams. For such a Young diagram we compute by \cite{GAP} the
associated character degree. There are two cases:

1) $b(\Al_n)$ is not $\chi_{\lambda_{up}}(1)$ and
$\chi_{\lambda_{dn}}(1)$. We have
\[
\chi_{\lambda_{up}}(1)^2+\chi_{\lambda_{dn}}(1)^2\geq
2\chi_{\lambda_{up}}(1)\chi_{\lambda_{dn}}(1)>\frac{\chi_\lambda(1)^2}{2}=\frac{b(\Sy_n)^2}{2}
\]
where the inequality in the middle comes from
Lemma~\ref{proposition-inequality}. Using two irreducible characters
lying under $\chi_\lambda$ as well, we obtain the desired
inequality.

2) $b(\Al_n)$ is either $\chi_{\lambda_{up}}(1)$ or
$\chi_{\lambda_{dn}}(1)$. In particular, the largest degree (among
the degrees we have computed) falls into either $Y_{up}$ or
$Y_{dn}$. Then we just check that the sum of the squares of all
other smaller degrees is bigger than the square of this largest
degree.
\end{proof}

Theorem~\ref{theorem-main-1} now is just a consequence of
Propositions~\ref{proposition-theorem1}
and~\ref{proposition-theorem11}.


\section{Theorem \ref{theorem-main-3}}\label{section5}

We will prove Theorem \ref{theorem-main-3} in this section. As the
main ideas are basically the same as those in
Sections~\ref{section3}, we will skip most of the details.

\begin{proof}[Proof of Theorem \ref{theorem-main-3}] Following the proofs of Corollary~\ref{corollary-Theorem2}
and Proposition~\ref{proposition-Theorem2}, we obtain
\[
\sum_{\twolineindex{\chi\in\irr(\Sy_n)}{\chi(1)< b(\Sy_n)}}
\hspace{-10pt}\chi(1)^2\geq\max\left\{\frac{|\s
M_1|}{16},\frac{(n-\sqrt{2n}-(|\s M_1|-1))^2}{2n}\right\}
b(\Sy_n)^2,
\]
which implies that
\[
\varepsilon(\Sy_n)\geq \max\left\{\frac{|\s
M_1|}{16},\frac{(n-\sqrt{2n}-(|\s M_1|-1))^2}{2n}\right\}.
\]
It now easily follows that $\varepsilon(\Sy_n)\rightarrow\infty$ as
$n\rightarrow\infty$.

To estimate $\varepsilon(\Al_n)$, we again consider two cases. If
$b(\Al_n)=b(\Sy_n)$ we would have
\[
\varepsilon(\Al_n)\geq \frac{1}{2}\varepsilon(\Sy_n)
\]
and therefore there is nothing more to prove.

So from now on we assume that $b(\Al_n)<b(\Sy_n)$. Let $x$ be the
number of irreducible characters of $\Sy_n$ of degree bigger than
$b(\Al_n)$. These characters produce $2x$ irreducible characters of
$\Al_n$ of degree at least $b(\Al_n)/2$ and therefore
\[
\sum_{\twolineindex{\chi\in\irr(\Al_n)}{\chi(1)< b(\Al_n)}}
\hspace{-10pt}\chi(1)^2\geq 2x\cdot
\left(\frac{b(\Al_n)}{2}\right)^2,
\]
which yield
\begin{equation}\label{equation1}\varepsilon(\Al_n)\geq \frac{x}{2}.
\end{equation}

Let $y$ be the multiplicity of the character degree $b(\Al_n)$ of
$\Sy_n$. Then we have
\[
  \left|\left\{\nu\in\s P\,\Big|\,
  \frac{b(\Al_n)}{4}< \chi_\nu(1)<
  b(\Al_n)\right\}\right|\geq y-4x.
  \]
Each $\nu$ in this set produces either one irreducible character of
$\Al_n$ of degree greater than $b(\Al_n)/4$ or two irreducible
characters of $\Al_n$ of degree greater than $b(\Al_n)/8$. Thus
\begin{equation}\label{equation2}\varepsilon(\Al_n)\geq \frac{y-4x}{32}.
\end{equation}

On the other hand, by following similar arguments as in the proof of
Proposition~\ref{proposition-Theorem1-case21}, we get
\begin{equation}\label{equation3}\varepsilon(\Al_n)\geq \frac{(n-\sqrt{2n}-2x-(y-1))^2}{2n}.
\end{equation}

Now combining Equations~\ref{equation1}, \ref{equation2},
and~\ref{equation3}, we have
\[
\varepsilon(\Al_n)\geq \max \left\{
\frac{x}{2},\frac{y-4x}{32},\frac{(n-\sqrt{2n}-2x-(y-1))^2}{2n}\right\}.
\]
From this it is clear that $\varepsilon(\Al_n)\rightarrow\infty$ as
$n\rightarrow\infty$ and the proof is complete.
\end{proof}

\section*{Acknowledgement} The authors are grateful to Attila Mar\'oti for
several helpful comments.

\end{document}